\documentclass[12pt]{amsart}
\usepackage{amsmath,amssymb,amsthm}
\usepackage{enumitem}
\usepackage{mathtools}
\usepackage[numbers,sort&compress]{natbib}
\usepackage{hyperref}
\DeclareMathOperator{\hdim}{\dim_H}
\DeclareMathOperator{\mul}{mul}
\theoremstyle{plain}
\newtheorem{theorem}{Theorem}[section]
\newtheorem{lemma}[theorem]{Lemma}
\newtheorem{proposition}[theorem]{Proposition}
\newtheorem{corollary}[theorem]{Corollary}
\theoremstyle{remark}
\newtheorem{remark}[theorem]{Remark}
\counterwithin{equation}{section}
\allowdisplaybreaks
\begin{document}
\title[Upper and lower fluctuations]{Upper and lower fluctuations in Shallit's law of leap years in Pierce expansions}
\author{Min Woong Ahn}
\address{Department of Mathematics Education, Dankook University, 152, Jukjeon-ro, Suji-gu, Yongin-si, Gyeonggi-do 16890, Republic of Korea}
\email{mwahn@dankook.ac.kr}
\date{\today}
\subjclass[2020]{Primary 11K55; Secondary 28A80}
\keywords{leap year, Pierce expansion, upper and lower limit, exceptional set, Hausdorff dimension}
\begin{abstract}
Shallit's law of leap years in Pierce expansions gives, for Lebesgue-almost every $x\in[0,1]$, the upper and lower limits of a normalized difference between $Nx$ and the number of leap years up to the year $N$ determined by the Pierce expansion digits of $x$. In this paper, we show that, for every $x\in[0,1]$, both limits are determined by the lower limit of a normalized logarithm of the product of the first $n$ digits of $x$. In particular, the upper and lower limits always have the same absolute value. As applications, we characterize the sets of points with prescribed upper and lower limits, and show that each of these sets, if non-empty, is dense in $[0,1]$ and its intersection with any non-empty open subset of $[0,1]$ has full Hausdorff dimension. We also compare these sets with the sets defined by the growth rate of the digits, and show that, for each finite parameter, their difference has full Hausdorff dimension.
\end{abstract}
\maketitle
\tableofcontents
\section{Introduction} \label{Introduction}
Shallit \cite{Sha94} introduced a rule for determining leap years which generalizes both the Julian and Gregorian calendars. An \emph{intercalation sequence} is a finite or infinite sequence $\sigma=(\sigma_k)_{k\geq1}$ of positive integers satisfying $\sigma_k\geq2$ for all $k\geq2$. The year $N$ is declared to be a leap year if
\[
\sum_{k\geq1}(-1)^{k+1}\mul(N,\sigma_1\dotsm\sigma_k)=1,
\]
where $\mul(m,n)=1$ if $m$ is a multiple of $n$ and $\mul(m,n)=0$ otherwise. If $L(\sigma,N)$ denotes the number of leap years from year $1$ through year $N$, then Shallit proved the formula
\begin{align} \label{L formula}
L(\sigma,N)=\sum_{k\geq1}(-1)^{k+1}\left\lfloor\frac{N}{\sigma_1\dotsm\sigma_k}\right\rfloor.
\end{align}
The Pierce expansion of $x\in[0,1]$ is given by
\begin{align} \label{Pierce expansion}
x
&=\sum_{k\in\mathbb N}\frac{(-1)^{k+1}}{d_1(x)\dotsm d_k(x)} \nonumber\\
&=\frac1{d_1(x)}-\frac1{d_1(x)d_2(x)}+\frac1{d_1(x)d_2(x)d_3(x)}-\dotsb,
\end{align}
where the digits $d_k(x)$ are strictly increasing as long as they are finite (see Section \ref{Preliminaries} for the precise definition). As usual, the digit sequence of a rational point consists of finitely many finite digits followed by $\infty$'s, and we use the convention $1/\infty=0$. For a rational $x\neq0$ whose last finite digit is $d_m(x)$, we set $L((d_k(x))_{k\in\mathbb N},N)\coloneqq L((d_1(x),\dotsc,d_m(x)),N)$, and we set $L((d_k(0))_{k\in\mathbb N},N)\coloneqq0$. In other words, the terms in \eqref{L formula} involving an infinite digit are interpreted as zero. We refer the reader to \cite{Ahn23,Pie29,Sch95,Sha86} for basic facts about Pierce expansions.\par
For $x\in[0,1]$ and an integer $N\geq2$, put
\begin{align*}
Q_x(N)\coloneqq\frac{Nx-L((d_k(x))_{k\in\mathbb N},N)}{\sqrt{\log N}}.
\end{align*}
Taking the Pierce expansion digit sequence as an intercalation sequence, Shallit \cite[Theorem 3]{Sha94} proved that, for Lebesgue-almost every $x\in[0,1]$,
\begin{align} \label{Shallit law}
\limsup_{N\to\infty}Q_x(N)=\frac1{\sqrt2}
\quad\text{and}\quad
\liminf_{N\to\infty}Q_x(N)=-\frac1{\sqrt2}.
\end{align}
In \cite{Ahn25}, we considered, for each $\alpha\in[0,\infty]$, the set
\begin{align} \label{definition S alpha}
S(\alpha)\coloneqq\biggl\{x\in[0,1]:\ &\limsup_{N\to\infty}Q_x(N)=\frac1{\sqrt{2\alpha}} \text{ and } \liminf_{N\to\infty}Q_x(N)=-\frac1{\sqrt{2\alpha}}\biggr\},
\end{align}
with the conventions $1/0=\infty$ and $1/\infty=0$, and showed that $S(\alpha)$ is dense in $[0,1]$ and that
\[
\hdim(U\cap S(\alpha))=1
\]
for every non-empty open subset $U$ of $[0,1]$ and every $\alpha\in[0,\infty]$.\par
The proof in \cite{Ahn25} relied on the set
\begin{align} \label{definition A alpha}
A(\alpha)\coloneqq\left\{x\in[0,1]:\lim_{n\to\infty}\frac{\log d_n(x)}n=\alpha\right\},
\end{align}
defined for each $\alpha\in[0,\infty]$, which has Hausdorff dimension $1$ by \cite[Corollary 1.3]{Ahn24}, and on the inclusion $A(\alpha)\subseteq S(\alpha)$. It was also shown in \cite[Lemma 3.6]{Ahn25} that every $x\in A(\alpha)$ satisfies
\begin{align} \label{old product growth}
\lim_{n\to\infty}\frac{2\log(d_1(x)\dotsm d_n(x))}{n^2}=\alpha.
\end{align}
The inclusion $A(\alpha)\subseteq S(\alpha)$ was sufficient for determining the Hausdorff dimension, but it does not characterize $S(\alpha)$. In particular, it does not tell us whether there exist points at which the two limits in \eqref{definition S alpha} have different absolute values. In view of \eqref{old product growth}, it is natural to consider the lower limit of the quotient in \eqref{old product growth} without assuming the existence of the limit.\par
In this paper, we answer these questions. For $x\in[0,1]$, define
\begin{align} \label{definition gamma}
\gamma(x)\coloneqq\liminf_{n\to\infty}
\frac{2\log(d_1(x)\dotsm d_n(x))}{n^2},
\end{align}
with the convention $\log\infty=\infty$. Our main result shows that $\gamma(x)$ determines both the upper and lower limits of $Q_x(N)$.
\begin{theorem} \label{main theorem}
For every $x\in[0,1]$, we have
\begin{align} \label{main formula}
\limsup_{N\to\infty}Q_x(N)=\frac1{\sqrt{2\gamma(x)}}
\quad\text{and}\quad
\liminf_{N\to\infty}Q_x(N)=-\frac1{\sqrt{2\gamma(x)}},
\end{align}
with the conventions $1/\sqrt0=\infty$ and $1/\sqrt\infty=0$.
\end{theorem}
In particular, for every $x\in[0,1]$,
\[
\limsup_{N\to\infty}Q_x(N)
=-\liminf_{N\to\infty}Q_x(N),
\]
i.e., the upper and lower limits of $Q_x(N)$ always have the same absolute value.\par
As a first consequence of Theorem \ref{main theorem}, we obtain the following characterization of the sets $S(\alpha)$.
\begin{corollary}\label{S alpha characterization}
For each $\alpha\in[0,\infty]$, we have
\[
S(\alpha)=\{x\in[0,1]:\gamma(x)=\alpha\}.
\]
In particular, this set is dense in $[0,1]$, and its intersection with any non-empty open subset of $[0,1]$ has Hausdorff dimension $1$.
\end{corollary}
For $a,b\in[-\infty,\infty]$, define
\begin{align*}
E(a,b)\coloneqq\left\{x\in[0,1]:
\limsup_{N\to\infty}Q_x(N)=a
\text{ and }
\liminf_{N\to\infty}Q_x(N)=b
\right\}.
\end{align*}
The following corollary, which follows from Theorem \ref{main theorem} and Corollary \ref{S alpha characterization}, completely determines which of the sets $E(a,b)$ are non-empty.
\begin{corollary} \label{pair classification}
Let $a,b\in[-\infty,\infty]$. Then $E(a,b)$ is non-empty if and only if $a=-b\in[0,\infty]$. More precisely, if $u\in(0,\infty)$, then
\[
E(u,-u)=S\left(\frac1{2u^2}\right),
\]
while $E(0,0)=S(\infty)$ and $E(\infty,-\infty)=S(0)$. Consequently, every non-empty set $E(a,b)$ is dense in $[0,1]$, and its intersection with any non-empty open subset of $[0,1]$ has Hausdorff dimension $1$.
\end{corollary}
By \eqref{old product growth} and Corollary \ref{S alpha characterization}, we have $A(\alpha)\subseteq S(\alpha)$ for each $\alpha\in[0,\infty]$. The following theorem shows that, for each finite $\alpha$, the difference $S(\alpha)\setminus A(\alpha)$ is still large in the Hausdorff dimension sense.
\begin{theorem} \label{difference theorem}
For each $\alpha\in[0,\infty)$, we have
\[
\hdim\bigl(S(\alpha)\setminus A(\alpha)\bigr)=1.
\]
In particular, $A(\alpha)\subsetneq S(\alpha)$ for each $\alpha\in[0,\infty)$. Moreover, $A(\infty)=S(\infty)$.
\end{theorem}
We mention some recent results related to the growth of Pierce expansion digits. Lu and Long \cite{LuLong25} studied level sets defined by the lower and upper limits of the normalized growth of the digits, and Long, Liang, and Lu \cite{LongLiangLu25} investigated limsup and liminf threshold sets for the digits. Long, Lu, and Liang \cite{LongLuLiang25} determined the Hausdorff dimension of joint liminf--limsup level sets for fluctuations of $\log d_n(x)-n$ of sublinear order. Note that these sets are defined directly in terms of individual digits. On the other hand, as Theorem \ref{main theorem} shows, the upper and lower limits of $Q_x(N)$ depend only on the growth of the products $d_1(x)\dotsm d_n(x)$.\par
We briefly describe the idea of the proof of Theorem \ref{main theorem}. For the upper bound of $|Q_x(N)|$, we use an estimate which holds for all $N$ (Lemma \ref{global bound lemma}). For the reverse inequalities, we use the integers introduced by Shallit in \cite[Theorem 2]{Sha94} (see Lemma \ref{canonical excursion lemma}). We remark that the argument uses only the lower limit in \eqref{definition gamma}, and not the existence of the limit of $(\log d_n(x))/n$ as $n\to\infty$.\par
This paper is organized as follows. In Section \ref{Preliminaries}, we recall some elementary facts about Pierce expansions and some earlier results which will be used later. In Section \ref{Auxiliary results}, we establish auxiliary results needed in the proof of Theorem \ref{main theorem}. We prove Theorem \ref{main theorem} in Section \ref{Proof main theorem}. In Section \ref{Consequences}, we prove Corollaries \ref{S alpha characterization} and \ref{pair classification} and Theorem \ref{difference theorem}.\par
Throughout the paper, we denote by $\mathbb N$ the set of positive integers, by $\mathbb N_\infty\coloneqq\mathbb N\cup\{\infty\}$ the set of extended positive integers, and by $\mathbb I\coloneqq[0,1]\setminus\mathbb Q$ the set of irrational numbers in $[0,1]$. For $y\in\mathbb R$, we write $\{y\}\coloneqq y-\lfloor y\rfloor$ for the fractional part of $y$. The closed unit interval $[0,1]$ is endowed with the usual topology. Following the usual convention, we define $c\cdot\infty\coloneqq\infty$ and $c/0\coloneqq\infty$ for any $c>0$, $1/\infty\coloneqq0$, and $\log\infty\coloneqq\infty$. The logarithm $\log$ always means the natural logarithm.
\section{Preliminaries} \label{Preliminaries}
In this section, we introduce some elementary facts about Pierce expansions and some results from \cite{Ahn24,Ahn25} which will be used later.\par
Let $d_1:[0,1]\to\mathbb N_\infty$ and $T:[0,1]\to[0,1]$ be given by
\[
d_1(x)\coloneqq
\begin{cases}
\lfloor1/x\rfloor,&\text{if } x\neq0,\\
\infty,&\text{if } x=0,
\end{cases}
\quad\text{and}\quad
T(x)\coloneqq
\begin{cases}
1-d_1(x)x,&\text{if } x\neq0,\\
0,&\text{if } x=0,
\end{cases}
\]
respectively, and define $d_n(x)\coloneqq d_1(T^{n-1}(x))$ for each $n\in\mathbb N$. When a point $x\in[0,1]$ is fixed, we write $d_n$ for $d_n(x)$ whenever no confusion can arise. Then
\[
d_{n+1}(x)\geq d_n(x)+1
\]
for each $n\in\mathbb N$ such that $d_n(x)$ is finite. In particular,
\begin{align} \label{digit lower bound}
d_n(x)\geq n
\end{align}
for all $n\in\mathbb N$. If $x\in\mathbb I$, then $(d_n(x))_{n\in\mathbb N}$ is a strictly increasing sequence of positive integers. Conversely, every strictly increasing sequence of positive integers is the Pierce expansion digit sequence of a unique irrational number in $[0,1]$. If $x$ is rational, then its digit sequence consists of finitely many finite digits followed by $\infty$'s. See, e.g., \cite[Section 2]{Ahn23}.\par
We shall use the following three results from \cite{Ahn24,Ahn25}.
\begin{proposition}[{\cite[Lemma 3.6]{Ahn25}}] \label{old product growth proposition}
Let $\alpha\in[0,\infty]$. If $x\in A(\alpha)$, then
\[
\lim_{n\to\infty}\frac{2\log(d_1(x)\dotsm d_n(x))}{n^2}=\alpha.
\]
\end{proposition}
The following proposition is a special case of \cite[Theorem 1.1]{Ahn24}.
\begin{proposition}[{\cite[Theorem 1.1]{Ahn24}}] \label{Ahn24 full dimension proposition}
Let $\varphi:\mathbb N\to(0,\infty)$ be a non-decreasing function such that
\[
\frac{\varphi(n)}{\log n}\to\infty
\quad\text{and}\quad
\frac{\varphi(n+1)}{\sum_{k=1}^n\varphi(k)}\to0
\]
as $n\to\infty$. Put
\[
F(\varphi)\coloneqq\left\{x\in(0,1]:\lim_{n\to\infty}\frac{\log d_n(x)}{\varphi(n)}=1\right\}.
\]
Then $\hdim F(\varphi)=1$.
\end{proposition}
\begin{proposition}[{\cite[Theorem 1.4]{Ahn25}}] \label{old S theorem}
For each $\alpha\in[0,\infty]$, the set $S(\alpha)$ is dense in $[0,1]$. Moreover,
\[
\hdim(U\cap S(\alpha))=1
\]
for every non-empty open subset $U$ of $[0,1]$.
\end{proposition}
\section{Auxiliary results} \label{Auxiliary results}
In this section, we establish some auxiliary results which will be used in the proof of Theorem \ref{main theorem}.\par
The following lemma expresses the numerator of $Q_x(N)$ in terms of fractional parts and gives an upper bound for it which holds for all $N$.
\begin{lemma} \label{global bound lemma}
Let $x\in[0,1]$ and $N\in\mathbb N$. Then
\begin{align} \label{fractional formula}
Nx-L((d_k)_{k\in\mathbb N},N)
=\sum_{n\in\mathbb N}(-1)^{n+1}
\left\{\frac{N}{d_1\dotsm d_n}\right\}.
\end{align}
Moreover, if $r\in\mathbb N$ satisfies $d_1\dotsm d_r>N$, then
\begin{align*}
\left|Nx-L((d_k)_{k\in\mathbb N},N)\right|\leq \frac r2+1.
\end{align*}
\end{lemma}
\begin{proof}
Subtracting \eqref{L formula}, with $\sigma_k=d_k$, from $N$ times \eqref{Pierce expansion}, we obtain \eqref{fractional formula}. Note that the series in \eqref{fractional formula} is well-defined. Indeed, if $x$ is rational, then all but finitely many terms vanish; if $x$ is irrational, then $d_1\dotsm d_n\geq n!$ by \eqref{digit lower bound}.\par
Now, suppose that $d_1\dotsm d_r>N$. Then
\[
\left\{\frac{N}{d_1\dotsm d_n}\right\}
=\frac{N}{d_1\dotsm d_n}
\]
for all $n\geq r$, and so
\begin{align*}
Nx-L((d_k)_{k\in\mathbb N},N)
=\sum_{n=1}^{r-1}(-1)^{n+1}\left\{\frac{N}{d_1\dotsm d_n}\right\}+N\sum_{n=r}^\infty\frac{(-1)^{n+1}}{d_1\dotsm d_n}.
\end{align*}
Since each fractional part lies in $[0,1)$, the absolute value of the first sum is at most $\lceil(r-1)/2\rceil\leq r/2$. The second sum is an alternating series whose terms decrease in absolute value, and hence its absolute value is less than $1$. This proves the inequality.
\end{proof}
The next lemma shows that the lower limit in \eqref{definition gamma} is also the lower limit along the even indices and along the odd indices.
\begin{lemma} \label{parity lemma}
Let $x\in\mathbb I$. Then
\[
\liminf_{r\to\infty}
\frac{2\sum_{k=1}^{2r}\log d_k}{(2r)^2}
=
\liminf_{r\to\infty}
\frac{2\sum_{k=1}^{2r+1}\log d_k}{(2r+1)^2}
=\gamma(x).
\]
\end{lemma}
\begin{proof}
Since $\sum_{k=1}^n\log d_k$ is increasing in $n$, we have
\[
\frac{2\sum_{k=1}^{2r+1}\log d_k}{(2r+1)^2}
\geq
\left(\frac{2r}{2r+1}\right)^2
\frac{2\sum_{k=1}^{2r}\log d_k}{(2r)^2}.
\]
Hence the lower limit along the odd indices is at least the lower limit along the even indices. Similarly,
\[
\frac{2\sum_{k=1}^{2r}\log d_k}{(2r)^2}
\geq
\left(\frac{2r-1}{2r}\right)^2
\frac{2\sum_{k=1}^{2r-1}\log d_k}{(2r-1)^2},
\]
which gives the reverse inequality. Thus the two lower limits are equal. But the lower limit of the full sequence is the minimum of these two lower limits, and so both of them are equal to $\gamma(x)$ by \eqref{definition gamma}.
\end{proof}
We now recall the integers used by Shallit in \cite[Theorem 2]{Sha94}. For $x\in\mathbb I$ and $j\in\mathbb N$, put
\begin{align} \label{N j definition}
N_j(x)\coloneqq-1+d_1(x)-d_1(x)d_2(x)+\dotsb+(-1)^{j+1}d_1(x)\dotsm d_j(x),
\end{align}
and put $N_0(x)\coloneqq-1$. Then $N_{2r+1}(x)>0$ for each $r\in\mathbb N$. For each integer $r\geq2$, put
\[
M_{2r}(x)\coloneqq-N_{2r}(x)>0.
\]
When $x$ is fixed, we write $N_j$ for $N_j(x)$ and $M_{2r}$ for $M_{2r}(x)$. For any $x\in\mathbb I$, we have
\begin{align} \label{N M less D}
0<N_{2r+1}<d_1\dotsm d_{2r+1}
\quad\text{and}\quad
0<M_{2r}<d_1\dotsm d_{2r}\quad(r\geq2).
\end{align}
These inequalities follow by induction from \eqref{N j definition} and the strict monotonicity of the digits. Moreover, for each $r\in\mathbb N$,
\[
N_{2r+1}=N_{2r-1}+d_1\dotsm d_{2r}(d_{2r+1}-1)>N_{2r-1},
\]
and, for each integer $r\geq3$,
\[
M_{2r}=M_{2r-2}+d_1\dotsm d_{2r-1}(d_{2r}-1)>M_{2r-2}.
\]
Hence $N_{2r+1}\to\infty$ and $M_{2r}\to\infty$ as $r\to\infty$.\par
The lower bounds in the following lemma follow from \cite[Theorem 2]{Sha94}. We also need the upper bounds, which we prove in a similar manner.
\begin{lemma} \label{canonical excursion lemma}
Let $x\in\mathbb I$. Then the following hold.
\begin{enumerate}[label=\upshape(\roman*), ref=(\roman*), leftmargin=*, widest=ii]
\item \label{canonical excursion lemma 1}
If $\sum_{n\in\mathbb N}1/d_n<\infty$, then, for each integer $r\geq2$,
\begin{align} \label{canonical summable positive}
N_{2r+1}x-L((d_k)_{k\in\mathbb N},N_{2r+1})
\geq r+1-\sum_{n\in\mathbb N}\frac1{d_n}
\end{align}
and
\begin{align} \label{canonical summable negative}
M_{2r}x-L((d_k)_{k\in\mathbb N},M_{2r})
\leq-r+2+\sum_{n\in\mathbb N}\frac1{d_n}.
\end{align}
\item \label{canonical excursion lemma 2}
For each integer $r\geq2$,
\begin{align} \label{canonical crude positive}
N_{2r+1}x-L((d_k)_{k\in\mathbb N},N_{2r+1})\geq\frac r2
\end{align}
and
\begin{align} \label{canonical crude negative}
M_{2r}x-L((d_k)_{k\in\mathbb N},M_{2r})
\leq-\frac r2+\frac32+\frac1{2r+1}.
\end{align}
\end{enumerate}
\end{lemma}
\begin{proof}
By \cite[Theorem 2]{Sha94}, we have
\begin{align} \label{positive Shallit estimate}
N_{2r+1}x-L((d_k)_{k\in\mathbb N},N_{2r+1}) \geq
\sum_{j=1}^{r+1}
\left(1-\frac1{d_{2j-1}}\right)
\left(1-\frac1{d_{2j}}\right).
\end{align}
We now establish an analogous upper bound for $M_{2r}$. By \eqref{N j definition}, for $1\leq i\leq j$,
\begin{align} \label{residue N}
N_j\equiv N_{i-1}\pmod{d_1\dotsm d_i}.
\end{align}
Let $r\geq2$. The contribution of the first two terms in \eqref{fractional formula}, with $N=M_{2r}$, is at most $1$. For $2\leq j\leq r$, it follows from \eqref{residue N} and the signs of $N_{2j-2}$ and $N_{2j-1}$ that
\[
\left\{\frac{M_{2r}}{d_1\dotsm d_{2j-1}}\right\}
=-\frac{N_{2j-2}}{d_1\dotsm d_{2j-1}}
\quad\text{and}\quad
\left\{\frac{M_{2r}}{d_1\dotsm d_{2j}}\right\}
=1-\frac{N_{2j-1}}{d_1\dotsm d_{2j}}.
\]
Since
\[
N_{2j-1}=N_{2j-2}+d_1\dotsm d_{2j-1},
\]
we find that
\begin{align*}
\left\{\frac{M_{2r}}{d_1\dotsm d_{2j-1}}\right\}
-\left\{\frac{M_{2r}}{d_1\dotsm d_{2j}}\right\}
=-\left(1-\frac1{d_{2j}}\right)
\left(1+\frac{N_{2j-2}}{d_1\dotsm d_{2j-1}}\right).
\end{align*}
But $N_{2j-2}\geq-d_1\dotsm d_{2j-2}$, and hence
\begin{align*}
\left\{\frac{M_{2r}}{d_1\dotsm d_{2j-1}}\right\}
-\left\{\frac{M_{2r}}{d_1\dotsm d_{2j}}\right\}
\leq-
\left(1-\frac1{d_{2j-1}}\right)
\left(1-\frac1{d_{2j}}\right).
\end{align*}
By \eqref{N M less D}, the alternating tail in \eqref{fractional formula} starting from the term with index $2r+1$ is positive and less than $1/d_{2r+1}$. Therefore,
\begin{align} \label{negative canonical estimate}
M_{2r}x-L((d_k)_{k\in\mathbb N},M_{2r})
\leq
1-\sum_{j=2}^{r}
\left(1-\frac1{d_{2j-1}}\right)
\left(1-\frac1{d_{2j}}\right)
+\frac1{d_{2r+1}}.
\end{align}
(i) Suppose that $\sum_{n\in\mathbb N}1/d_n<\infty$. Since
\[
\left(1-\frac1{d_{2j-1}}\right)
\left(1-\frac1{d_{2j}}\right)
\geq1-\frac1{d_{2j-1}}-\frac1{d_{2j}}
\]
for each $j\in\mathbb N$, it follows from \eqref{positive Shallit estimate} that
\[
N_{2r+1}x-L((d_k)_{k\in\mathbb N},N_{2r+1})
\geq r+1-\sum_{n\in\mathbb N}\frac1{d_n},
\]
which is \eqref{canonical summable positive}. Similarly, \eqref{negative canonical estimate} gives
\begin{align*}
M_{2r}x-L((d_k)_{k\in\mathbb N},M_{2r})
\leq-r+2+
\sum_{n=3}^{2r}\frac1{d_n}+\frac1{d_{2r+1}}
\leq-r+2+\sum_{n\in\mathbb N}\frac1{d_n},
\end{align*}
which proves \eqref{canonical summable negative}.\par
(ii) By \eqref{digit lower bound}, for each integer $j\geq2$,
\[
\left(1-\frac1{d_{2j-1}}\right)
\left(1-\frac1{d_{2j}}\right)
\geq\left(1-\frac13\right)\left(1-\frac14\right)=\frac12.
\]
Applying this to the terms with $2\leq j\leq r+1$ in \eqref{positive Shallit estimate}, we obtain \eqref{canonical crude positive}. Also, by \eqref{negative canonical estimate} and $d_{2r+1}\geq2r+1$, we have
\[
M_{2r}x-L((d_k)_{k\in\mathbb N},M_{2r})
\leq1-\frac{r-1}{2}+\frac1{2r+1},
\]
which is \eqref{canonical crude negative}. This completes the proof of the lemma.
\end{proof}
\section{Proof of the main theorem} \label{Proof main theorem}
\begin{proof}[Proof of Theorem \ref{main theorem}]
We first consider the case where $x$ is rational. Then $d_r=\infty$ for some $r\in\mathbb N$, and so Lemma \ref{global bound lemma} gives
\[
\left|Nx-L((d_k)_{k\in\mathbb N},N)\right|\leq\frac r2+1
\]
for all $N\in\mathbb N$. Hence both limits in \eqref{main formula} are $0$. Since $\gamma(x)=\infty$ by \eqref{definition gamma}, this proves \eqref{main formula} for rational $x$.\par
We may therefore assume that $x\in\mathbb I$. Put $\gamma\coloneqq\gamma(x)$. We consider three cases.\par
\emph{Case I}: $0<\gamma<\infty$. Let $\varepsilon\in(0,\gamma)$ and $\delta>0$ be arbitrary. By the definition of $\gamma$, there exists $n_0\in\mathbb N$ such that
\begin{align} \label{sum lower finite gamma}
\sum_{k=1}^n\log d_k\geq\frac{\gamma-\varepsilon}{2}n^2
\end{align}
for all $n\geq n_0$. For sufficiently large $N$, put
\[
r(N)\coloneqq\left\lceil(1+\delta)\sqrt{\frac{2\log N}{\gamma-\varepsilon}}\right\rceil.
\]
Then \eqref{sum lower finite gamma} gives
\[
\log(d_1\dotsm d_{r(N)})>\log N,
\]
i.e., $d_1\dotsm d_{r(N)}>N$. By Lemma \ref{global bound lemma}, we have
\[
|Q_x(N)|
\leq\frac{r(N)/2+1}{\sqrt{\log N}}.
\]
By letting $N\to\infty$, then $\delta\to0^+$, and finally $\varepsilon\to0^+$, we obtain
\begin{align} \label{global squeeze finite gamma}
-\frac1{\sqrt{2\gamma}}
\leq\liminf_{N\to\infty}Q_x(N)
\leq\limsup_{N\to\infty}Q_x(N)
\leq\frac1{\sqrt{2\gamma}}.
\end{align}
Moreover, since
\[
\sum_{k=1}^n\log d_k\leq n\log d_n
\]
for each $n\in\mathbb N$, it follows from \eqref{sum lower finite gamma} that $\log d_n\geq(\gamma-\varepsilon)n/2$ for all sufficiently large $n$. Hence
\[
C\coloneqq\sum_{n\in\mathbb N}\frac1{d_n}<\infty.
\]
By Lemma \ref{parity lemma}, there is a sequence $(r_j)_{j\in\mathbb N}$ of positive integers with $r_j\to\infty$ such that
\[
\frac{2\sum_{k=1}^{2r_j+1}\log d_k}{(2r_j+1)^2}\to\gamma
\quad\text{as } j\to\infty.
\]
Note that $N_{2r_j+1}\to\infty$ as $j\to\infty$. For all sufficiently large $j$, we have $r_j+1-C>0$, while $N_{2r_j+1}<d_1\dotsm d_{2r_j+1}$ by \eqref{N M less D}, so that
\[
\log N_{2r_j+1}<\sum_{k=1}^{2r_j+1}\log d_k.
\]
Hence, by Lemma \ref{canonical excursion lemma}\ref{canonical excursion lemma 1}, we find that
\begin{align*}
\limsup_{N\to\infty}Q_x(N)
\geq\limsup_{j\to\infty}
Q_x(N_{2r_j+1})
\geq\lim_{j\to\infty}
\frac{r_j+1-C}{\sqrt{\sum_{k=1}^{2r_j+1}\log d_k}}
=\frac1{\sqrt{2\gamma}}.
\end{align*}
Together with \eqref{global squeeze finite gamma}, this proves the first equality in \eqref{main formula}.\par
Similarly, by Lemma \ref{parity lemma}, there is a sequence $(s_j)_{j\in\mathbb N}$ of positive integers with $s_j\to\infty$ such that
\begin{align} \label{even gamma subsequence}
\frac{2\sum_{k=1}^{2s_j}\log d_k}{(2s_j)^2}\to\gamma
\quad\text{as } j\to\infty.
\end{align}
Note that $M_{2s_j}\to\infty$ as $j\to\infty$. For all sufficiently large $j$, we have $-s_j+2+C<0$, while $M_{2s_j}<d_1\dotsm d_{2s_j}$ by \eqref{N M less D}, so that
\[
\log M_{2s_j}<\sum_{k=1}^{2s_j}\log d_k.
\]
Therefore, by Lemma \ref{canonical excursion lemma}\ref{canonical excursion lemma 1},
\[
Q_x(M_{2s_j})
\leq
\frac{-s_j+2+C}{\sqrt{\sum_{k=1}^{2s_j}\log d_k}}
\]
for all sufficiently large $j$. Using \eqref{even gamma subsequence}, we obtain
\[
\liminf_{N\to\infty}Q_x(N)
\leq-\frac1{\sqrt{2\gamma}}.
\]
Together with \eqref{global squeeze finite gamma}, this proves the second equality in \eqref{main formula}.\par
\emph{Case II}: $\gamma=0$. By Lemma \ref{parity lemma}, there are sequences $(r_j)_{j\in\mathbb N}$ and $(s_j)_{j\in\mathbb N}$ of positive integers with $r_j\to\infty$ and $s_j\to\infty$ such that
\[
\frac{\sum_{k=1}^{2r_j+1}\log d_k}{(2r_j+1)^2}\to0
\quad\text{and}\quad
\frac{\sum_{k=1}^{2s_j}\log d_k}{(2s_j)^2}\to0
\]
as $j\to\infty$. Since $N_{2r_j+1}\to\infty$ as $j\to\infty$ and $N_{2r_j+1}<d_1\dotsm d_{2r_j+1}$, Lemma \ref{canonical excursion lemma}\ref{canonical excursion lemma 2} gives
\[
Q_x(N_{2r_j+1})
\geq
\frac{r_j/2}{\sqrt{\sum_{k=1}^{2r_j+1}\log d_k}}\to\infty
\quad\text{as } j\to\infty.
\]
Thus the upper limit in \eqref{main formula} is $\infty$. Similarly, since $M_{2s_j}\to\infty$ as $j\to\infty$, $M_{2s_j}<d_1\dotsm d_{2s_j}$, and $-s_j/2+3/2+1/(2s_j+1)<0$ for all sufficiently large $j$, Lemma \ref{canonical excursion lemma}\ref{canonical excursion lemma 2} gives
\[
Q_x(M_{2s_j})
\leq
\frac{-s_j/2+3/2+1/(2s_j+1)}{\sqrt{\sum_{k=1}^{2s_j}\log d_k}}\to-\infty
\quad\text{as } j\to\infty.
\]
Hence the lower limit in \eqref{main formula} is $-\infty$.\par
\emph{Case III}: $\gamma=\infty$. Let $K>0$ and $\delta>0$ be arbitrary. Then
\[
\sum_{k=1}^n\log d_k\geq\frac K2n^2
\]
for all sufficiently large $n$. For sufficiently large $N$, put
\[
r(N)\coloneqq\left\lceil(1+\delta)\sqrt{\frac{2\log N}{K}}\right\rceil.
\]
Then $d_1\dotsm d_{r(N)}>N$, and so Lemma \ref{global bound lemma} gives
\[
\limsup_{N\to\infty}
|Q_x(N)|
\leq\frac{1+\delta}{\sqrt{2K}}.
\]
By letting $\delta\to0^+$ and then $K\to\infty$, we conclude that both limits in \eqref{main formula} are $0$. This completes the proof of the theorem.
\end{proof}
\section{Proofs of the other results} \label{Consequences}
In this section, we prove Corollaries \ref{S alpha characterization} and \ref{pair classification} and Theorem \ref{difference theorem}.
\begin{proof}[Proof of Corollary \ref{S alpha characterization}]
Let $\alpha\in[0,\infty]$. By Theorem \ref{main theorem} and the definition \eqref{definition S alpha}, for each $x\in[0,1]$, we have
\[
x\in S(\alpha)
\quad\Longleftrightarrow\quad
\frac1{\sqrt{2\gamma(x)}}=\frac1{\sqrt{2\alpha}}
\quad\Longleftrightarrow\quad
\gamma(x)=\alpha,
\]
with the conventions adopted above. This proves the desired equality. The statements on the denseness and the Hausdorff dimension follow from Proposition \ref{old S theorem}.
\end{proof}
\begin{proof}[Proof of Corollary \ref{pair classification}]
By Theorem \ref{main theorem}, the upper and lower limits of $Q_x(N)$ are negatives of each other for every $x\in[0,1]$. Hence $E(a,b)$ is empty whenever $a\neq-b$.\par
If $0<u<\infty$, then $u=1/\sqrt{2\alpha}$ with $\alpha\coloneqq1/(2u^2)$, and so $E(u,-u)=S(1/(2u^2))$ by the definition \eqref{definition S alpha}. Similarly, we have $E(0,0)=S(\infty)$ and $E(\infty,-\infty)=S(0)$. In particular, $E(a,b)$ is non-empty whenever $a=-b\in[0,\infty]$. The statements on the denseness and the Hausdorff dimension follow from Corollary \ref{S alpha characterization}.
\end{proof}
Before proving Theorem \ref{difference theorem}, we note that, by Proposition \ref{old product growth proposition} and Corollary \ref{S alpha characterization},
\[
A(\alpha)\subseteq S(\alpha)
\]
for each $\alpha\in[0,\infty]$. This recovers \cite[Lemma 4.1]{Ahn25}.
\begin{proof}[Proof of Theorem \ref{difference theorem}]
We first consider the case $0<\alpha<\infty$. Fix $\beta>\alpha$. Choose a strictly increasing sequence $(m_j)_{j\in\mathbb N}$ of integers with $m_1\geq2$ such that
\begin{align} \label{m growth positive alpha}
m_{j+1}\geq m_j^3
\quad\text{and}\quad
m_{j+1}>4(\beta/\alpha)m_j
\end{align}
for each $j\in\mathbb N$, and define a non-decreasing function $\varphi:\mathbb N\to(0,\infty)$ by
\[
\varphi(n)\coloneqq\max\{\alpha n,\ \beta m_j:m_j\leq n\}
\]
for each $n\in\mathbb N$. Then $\alpha n\leq\varphi(n)\leq\beta n$ for all $n\in\mathbb N$, so that $\varphi(n)/\log n\to\infty$ as $n\to\infty$. Moreover,
\[
0\leq\frac{\varphi(n+1)}{\sum_{k=1}^n\varphi(k)}
\leq\frac{\beta(n+1)}{\alpha n(n+1)/2}
=\frac{2\beta}{\alpha n}\to0
\quad\text{as } n\to\infty.
\]
Hence, by Proposition \ref{Ahn24 full dimension proposition}, we have
\begin{align} \label{F phi full dimension}
\hdim F(\varphi)=1.
\end{align}
Put $q_j\coloneqq\lfloor m_{j+1}/2\rfloor$ for each $j\in\mathbb N$. By \eqref{m growth positive alpha}, we have $q_j>(\beta/\alpha)m_j$ and $q_j<m_{j+1}$ for all sufficiently large $j$, and so
\[
\varphi(m_j)=\beta m_j
\quad\text{and}\quad
\varphi(q_j)=\alpha q_j.
\]
Therefore, every $x\in F(\varphi)$ satisfies
\[
\frac{\log d_{m_j}(x)}{m_j}\to\beta
\quad\text{and}\quad
\frac{\log d_{q_j}(x)}{q_j}\to\alpha
\quad\text{as } j\to\infty,
\]
which implies that $x\notin A(\alpha)$.\par
We now determine $\gamma(x)$ for $x\in F(\varphi)$. Since $\varphi(k)\geq\alpha k$ for all $k\in\mathbb N$, we have
\[
\liminf_{n\to\infty}\frac{2}{n^2}\sum_{k=1}^n\varphi(k)\geq\alpha.
\]
On the other hand, by the definition of $\varphi$,
\[
\varphi(k)\leq\alpha k+\sum_{i:m_i\leq k}\max\{\beta m_i-\alpha k,0\}
\]
for each $k\in\mathbb N$. Here, the $i$th summand can be positive only if $k<(\beta/\alpha)m_i$, and it is at most $\beta m_i$. Hence
\[
\sum_{k=1}^{q_j}\varphi(k)
\leq\alpha\frac{q_j(q_j+1)}2
+\frac{\beta^2}{\alpha}\sum_{i=1}^j m_i^2.
\]
Since $m_{i+1}\geq m_i^3\geq2m_i$ for each $i\in\mathbb N$, we have
\[
\sum_{i=1}^j m_i^2<\frac43m_j^2.
\]
Moreover, $q_j\geq m_{j+1}/3\geq m_j^3/3$ for all sufficiently large $j$. Consequently,
\[
\frac{2}{q_j^2}\sum_{k=1}^{q_j}\varphi(k)
\leq\alpha\left(1+\frac1{q_j}\right)
+\frac{24\beta^2}{\alpha m_j^4}\to\alpha
\quad\text{as } j\to\infty.
\]
Thus
\begin{align} \label{phi quadratic alpha}
\liminf_{n\to\infty}\frac{2}{n^2}\sum_{k=1}^n\varphi(k)=\alpha.
\end{align}
Now, let $x\in F(\varphi)$ and $\varepsilon\in(0,1)$. Then
\[
(1-\varepsilon)\varphi(k)\leq\log d_k(x)\leq(1+\varepsilon)\varphi(k)
\]
for all sufficiently large $k$. The finitely many remaining terms make no contribution after division by $n^2$. Hence, by \eqref{phi quadratic alpha}, we have $(1-\varepsilon)\alpha\leq\gamma(x)\leq(1+\varepsilon)\alpha$. On letting $\varepsilon\to0^+$, we obtain $\gamma(x)=\alpha$, and so $x\in S(\alpha)$ by Corollary \ref{S alpha characterization}. This proves
\[
F(\varphi)\subseteq S(\alpha)\setminus A(\alpha).
\]
Therefore, by \eqref{F phi full dimension} and the monotonicity of the Hausdorff dimension, we conclude that
\[
\hdim\bigl(S(\alpha)\setminus A(\alpha)\bigr)=1.
\]\par
We next consider the case $\alpha=0$. Choose a strictly increasing sequence $(m_j)_{j\in\mathbb N}$ of integers with $m_1\geq2$ such that
\[
m_{j+1}\geq m_j^4
\quad\text{and}\quad
m_{j+1}>4m_j^2
\]
for each $j\in\mathbb N$, and define a non-decreasing function $\varphi:\mathbb N\to(0,\infty)$ by
\[
\varphi(n)\coloneqq\max\{\sqrt n,\ m_j:m_j\leq n\}
\]
for each $n\in\mathbb N$. Then $\sqrt n\leq\varphi(n)\leq n$ for all $n\in\mathbb N$, so that $\varphi(n)/\log n\to\infty$ as $n\to\infty$. Moreover,
\[
0\leq\frac{\varphi(n+1)}{\sum_{k=1}^n\varphi(k)}
\leq\frac{n+1}{\int_0^n\sqrt{t}\,dt}
=\frac{3(n+1)}{2n^{3/2}}\to0
\quad\text{as } n\to\infty.
\]
Hence $\hdim F(\varphi)=1$ by Proposition \ref{Ahn24 full dimension proposition}. Put $q_j\coloneqq\lfloor m_{j+1}/2\rfloor$ for each $j\in\mathbb N$. Then $q_j>m_j^2$ and $q_j<m_{j+1}$ for all sufficiently large $j$, and so
\[
\varphi(m_j)=m_j
\quad\text{and}\quad
\varphi(q_j)=\sqrt{q_j}.
\]
Therefore, every $x\in F(\varphi)$ satisfies
\[
\frac{\log d_{m_j}(x)}{m_j}\to1
\quad\text{and}\quad
\frac{\log d_{q_j}(x)}{q_j}\to0
\quad\text{as } j\to\infty,
\]
which implies that $x\notin A(0)$.\par
By the definition of $\varphi$,
\[
\varphi(k)\leq\sqrt{k}+\sum_{i:m_i\leq k}\max\{m_i-\sqrt{k},0\}
\]
for each $k\in\mathbb N$. Here, the $i$th summand can be positive only if $k<m_i^2$, and it is at most $m_i$. Hence
\[
\sum_{k=1}^{q_j}\varphi(k)
\leq q_j^{3/2}+\sum_{i=1}^j m_i^3
<q_j^{3/2}+\frac87m_j^3.
\]
Since $q_j\geq m_{j+1}/3\geq m_j^4/3$ for all sufficiently large $j$, it follows that
\[
0\leq\frac{2}{q_j^2}\sum_{k=1}^{q_j}\varphi(k)
\leq\frac{2}{\sqrt{q_j}}+\frac{144}{7m_j^5}\to0
\quad\text{as } j\to\infty.
\]
Thus
\[
\liminf_{n\to\infty}\frac{2}{n^2}\sum_{k=1}^n\varphi(k)=0.
\]
By the same argument as in the case $0<\alpha<\infty$, we have $\gamma(x)=0$ for every $x\in F(\varphi)$. Hence
\[
F(\varphi)\subseteq S(0)\setminus A(0),
\]
and therefore $\hdim(S(0)\setminus A(0))=1$.\par
Finally, we show that $A(\infty)=S(\infty)$. Since $A(\infty)\subseteq S(\infty)$ as noted before the proof, it suffices to show that $S(\infty)\subseteq A(\infty)$. Let $x\in S(\infty)$. Then $\gamma(x)=\infty$ by Corollary \ref{S alpha characterization}. Since the digits are non-decreasing, we have
\[
\frac{2}{n^2}\sum_{k=1}^n\log d_k\leq\frac{2\log d_n}{n}
\]
for each $n\in\mathbb N$. Hence $(\log d_n)/n\to\infty$ as $n\to\infty$, i.e., $x\in A(\infty)$. This completes the proof of the theorem.
\end{proof}
\begin{remark}
By Corollary \ref{S alpha characterization} and Theorem \ref{difference theorem}, for each $\alpha\in[0,\infty)$, the condition in \eqref{definition A alpha} is sufficient but not necessary for a point to belong to $S(\alpha)$, and the set of points in $S(\alpha)$ which do not satisfy this condition has full Hausdorff dimension. Thus, as far as the upper and lower limits of $Q_x(N)$ are concerned, the relevant quantity is $\gamma(x)$ in \eqref{definition gamma}, rather than the limit of $(\log d_n(x))/n$. Moreover, Corollary \ref{S alpha characterization} shows that $S(\alpha)$ is finite replacement-invariant for every $\alpha\in[0,\infty]$, in the sense of \cite{Ahn25}, since $\gamma(x)$ is unchanged under any well-defined finite replacement of the Pierce expansion digits. This answers the question left open in \cite[Remark 4.2]{Ahn25}. We also note that $A(1)$ has full Lebesgue measure by \cite[Theorem 16]{Sha86}, and that $\gamma(x)=1$ for every $x\in A(1)$ by Proposition \ref{old product growth proposition}. Therefore, Theorem \ref{main theorem} recovers Shallit's law \eqref{Shallit law}.
\end{remark}

\end{document}